\documentclass[a4paper]{article}

\usepackage{amsthm,amsmath,amsfonts,graphicx}
\usepackage[T1]{fontenc}
\usepackage{amsmath}
\usepackage{amssymb}
\usepackage[all]{xy}
\usepackage{bbm}
\usepackage{tikz}
\usepackage{cite}
\usepackage{tabularx}
\usepackage{caption}

\def\Z{\mathbb{Z}}
\def\N{\mathbb{N}}
\def\R{\mathbb{R}}

\DeclareMathOperator{\supp}{supp}

\newtheorem{theorem}{Theorem}
\newtheorem{conjecture}{Conjecture}

\newtheorem{definition}{Definition}
\newtheorem{proposition}{Proposition}

\newtheorem{lemma}{Lemma}

\newcolumntype{L}[1]{>{\raggedright\arraybackslash}p{#1}} 
\newcolumntype{C}[1]{>{\centering\arraybackslash}p{#1}} 
\newcolumntype{R}[1]{>{\raggedleft\arraybackslash}p{#1}} 

\title{The NL-flow polynomial}
\author{Barbara Altenbokum, Winfried Hochst\"attler, Johanna Wiehe \\ FernUniversit\"at in Hagen, Germany}
\date{}

\begin{document}

\maketitle

\begin{abstract}In 1982 V\'{i}ctor Neumann-Lara \cite{neumannlara}
  introduced the dichromatic number of a digraph $D$ as the smallest
  integer $k$ such that the vertices $V$ of $D$ can be colored with
  $k$ colors and each color class induces an acyclic digraph.
  In \cite{hochstarticle} a flow theory for the dichromatic number 
  transferring Tutte's theory of nowhere-zero flows (NZ-flows) from classic 
  graph colorings has been developed. The purpose of this paper is 
  to pursue this analogy by introducing a new definition of algebraic Neumann-Lara-flows 
  (NL-flows) and a closed formula for their polynomial.\\
  Furthermore we generalize the Equivalence Theorem for nowhere-zero flows 
  to NL-flows in the setting of regular oriented matroids. Finally we discuss
  computational aspects of computing the NL-flow polynomial
  for orientations of complete digraphs and obtain a closed formula in the
  acyclic case.
\end{abstract}
\setlength{\parindent}{0em}

\section{Introduction}

Large parts of graph theory have been driven by the Four Color Problem. In
particular it inspired William T.\ Tutte to develop his theory of
nowhere-zero flows \cite{tuttecon}.

In 1982 V\'{i}ctor Neumann-Lara \cite{neumannlara} introduced the
dichromatic number of a digraph $D$ as the smallest integer $k$ such
that the vertices $V$ of $D$ can be colored with $k$ colors and each
color class induces an acyclic digraph. Moreover, in 1985 he
conjectured, that every orientation of a simple planar graph can be
acyclically colored with two colors. This intrigueing problem led us
to trying to look for an analogy following Tutte's road map and develop a
corresponding flow theory, which we named {\sc Neumann-Lara}-flows (see \cite{hochstarticle}, \cite{altenbokum}).\\
First, we renew some definitions in order to simplify the notation in the forthcoming proofs. In Section $3$ we define the NL-flow polynomial. A short excursion to regular matroids yields the Equivalence Theorem for NL-flows in a general setting in Section $4$. Afterwards we discuss some computational aspects of the NL-flow polynomial for orientations of complete digraphs.\\
Our notation is fairly standard and, if not explicitely defined, should follow the books of Diestel \cite{diestel} for graphs and Bj\"orner et.\ al.\ \cite{BLSW} for oriented matroids. Note that all our digraphs may have parallel and antiparallel arcs.

\section{NL-flows and NL-coflows}

Let $D=(V,A)$ be a digraph and $G$ a finite Abelian group. A map $f:A\longrightarrow G$ is a ($G$-) flow in $D$ if it satisfies Kirchhoff's law of flow conservation 
\begin{flalign}\label{conservationcondition}
\sum_{a \in \partial^+(v)}f(a)=\sum_{a \in \partial^-(v)}f(a)
\end{flalign}

in every vertex $v \in V$. Let $n$ be the number of vertices, $m$ the number of arcs and $M$ the $n \times m $ -incidence matrix of $D$. We may identify $f$ with the $m \times 1$ -vector $f:=(f(a_1),...,f(a_m))^\top$ and with this notation, the conservation condition (\ref{conservationcondition}) is equivalent to the matrix equation
\begin{flalign}\label{Mf=0}
Mf=0,
\end{flalign}

where $0$ denotes the $n \times 1$ -zero vector. 

\begin{definition}
Let $D=(V,A)$ be a digraph and $G$ a finite Abelian group. An NL-$G$-flow in $D$ is a flow \mbox{$f: A \longrightarrow G$} in $D$ such that $D/supp(f)$ is totally cyclic. \\
For $k \geq 2$ and $G=\Z$, a flow $f$ is an NL-$k$-flow, if for all $a \in A$
$$ f(a) \in \lbrace 0, \pm 1,..., \pm(k-1)\rbrace,$$
satisfying (\ref{conservationcondition}) such that contracting its support yields a totally cyclic digraph.
\end{definition}

In the following we show that this definition is consistent with the one given in \cite{hochstarticle}, where a \textsc{Neumann-Lara} flow (NL-flow) is defined as a pair \mbox{$(f_1,f_2):A \longrightarrow \mathbb{Z}^2$} of flows related by the condition
\begin{flalign}\label{f1f2}
f_1(a)=0 \Rightarrow f_2(a)>0 \text{ , for all } a \in A.
\end{flalign}

\begin{definition}
Let $D=(V,A)$ denote a digraph. A set of arcs $S \subseteq A$ is a \emph{dijoin}, if $S$ intersects every directed cut.
\end{definition}

\begin{lemma}\label{dijoin}
$S \subseteq A$ is a dijoin if and only if $D/S$ is totally cyclic, i.e.\ every component is strongly connected.
\end{lemma}

\begin{proof}
If $D/S$ is totally cyclic, then it does not contain a directed cut. Hence $S$ must have intersected every directed cut. If $D/S$ is not totally cyclic it contains a directed cut, which is a directed cut in $D$ as well. Hence $S$ is not a dijoin.
\end{proof}

\begin{proposition}\label{supp dijoin}
Let $D=(V,A)$ be a digraph. A pair of flows $(f_1,f_2)$ is an NL-flow in $D$ if and only if supp$(f_1)$ is a dijoin.
\end{proposition}

\begin{proof}
Consider an NL-flow $(f_1,f_2)$ in $D$. By (\ref{f1f2}), the second component $f_2$ is strictly positive outside supp$(f_1)$. Thus, if we contract supp$(f_1)$, the restriction $f_2\vert_{A\setminus \supp(f_1)}$ is a strictly positive flow in the resulting digraph $D/$supp$(f_1)$. That is, every component of $D/$supp$(f_1)$ is strongly connected.\\
Conversely, let $f_1$ be a flow in $D$. If every component of $D/$supp$(f_1)$ is strongly connected, we certainly find a strictly positive flow \mbox{$\tilde{f_2}: A\setminus $supp$(f_1) \longrightarrow \mathbb{Z}$} in $D/$supp$(f_1)$, which in turn must be the restriction of some flow $f_2: A \longrightarrow \mathbb{Z}$ in $D$ using Linear Algebra. Combining $f_2$ and $f_1$, we have built an NL-flow $(f_1,f_2)$.\\
Lemma \ref{dijoin} completes the proof.
\end{proof}

With this definition we get an equivalence theorem in full analogy to the case of nowhere-zero flows (\cite{nesetril},\cite{tuttecon}).

\begin{theorem}\label{equiv.theorem}
Let $D=(V,A)$ be a digraph. Let $k \geq 2$ and $G$ be an Abelian group of order $k$. Then the following conditions are equivalent:
\begin{itemize}
\item[(i)] There exists an NL-$\mathbb{Z}_k$-flow in $D$.
\item[(ii)] There exists an NL-$G$-flow in $D$.
\item[(iii)] There exists an NL-$k$-flow in $D$.
\end{itemize}
\end{theorem}

We postpone the proof to Section \ref{equivtheo}, where it will be stated in the more general setting of regular matroids.\\

Now, recall that a map $g: A \longrightarrow G$ is a coflow in $D=(V,A)$ if it satisfies Kirchhoff's law for (weak) cycles
\begin{flalign}\label{coflowcondition}
\sum_{a \in C^+}g(a)=\sum_{a \in C^-}g(a).
\end{flalign}
Condition (\ref{coflowcondition}) is equivalent to the condition that the vector $g=(g(a_1),...,g(a_m))^\top$ is an element of the row space of M, that is $g=pM$, for some $1 \times n$ -vector $p \in G^V$. 

\begin{definition}
A \emph{feedback arc set} of a digraph $D=(V,A)$ is a set $S \subseteq A$ such that $D-S$ is acyclic. 
\end{definition}

\begin{definition}
Let $D=(V,A)$ be a digraph and $G$ a finite Abelian group. An NL-coflow in $D$ is a coflow $g: A \longrightarrow G$ in $D$ whose support contains a feedback arc set. For $k \in \mathbb{Z}$, a coflow $g$ is an NL-$k$-coflow if, in addition,
$$ g(a) \in \lbrace 0,\pm 1,...,\pm (k-1) \rbrace \text{ , for all } a \in A,$$
satisfying (\ref{coflowcondition}) such that its support contains a feedback arc set.
\end{definition}

Again, we will show that this definition is consistent with the one given in \cite{hochstarticle}, where a \textsc{Neumann-Lara} coflow (NL-coflow) is defined as a pair \mbox{$(g_1,g_2): A \longrightarrow \mathbb{Z}^2$} of coflows related by the condition
 \begin{flalign}\label{g1g2}
g_1(a)=0 \Rightarrow g_2(a)>0 \text{ for all } a \in A.
\end{flalign}

\begin{proposition}
Let $D=(V,A)$ be a digraph. A pair of coflows $(g_1,g_2)$ is an NL-coflow in $D$ if and only if supp$(g_1)$ contains a feedback arc set.
\end{proposition}

\begin{proof}
Consider an NL-coflow $(g_1,g_2)$. Condition (\ref{g1g2}) and (\ref{coflowcondition}) immediately yield that $D-$ supp$(g_1)$ is acyclic.\\
On the other hand, let $g_1$ be a coflow in $D$. If the subdigraph $D-$ supp$(g_1)$ is acyclic, we can find a strictly positive coflow $g_2$ in $D-$ supp$(g_1)$: We use a topological ordering $v_1,...,v_n$ of the vertices, define $p(v_i):=i$ and obtain a strictly positive vector $g=pM$, as pointed out in \cite{hochstarticle}. Combining $g_1$ and $g_2$, we have built an NL-coflow $(g_1,g_2)$.
\end{proof}

Recall (see e.g.\ \cite{bondy}) that a digraph $D=(V,A)$ admits a proper vertex coloring with $k$ colors if and only if there is a nowhere-zero-$k$-coflow $g:A \longrightarrow \mathbb{Z}_k$ in $D$.\\
Concerning acyclic vertex colorings and NL-coflows, a similar result is already obtained in \cite{hochstarticle}:

\begin{theorem}[Hochst\"attler]
A loopless digraph $D=(V,A)$ admits an acyclic vertex coloring with $k$ colors if and only if there is an NL-$k$-coflow in $D$.
\end{theorem}

Now, let us take a look at the planar case. Let $D=(V,A)$ be a plane digraph with plane dual $D^\ast=(V^\ast,A^\ast)$. For a map $f:A \longrightarrow \mathbb{Z}$, define a map \mbox{$f^\ast: A^\ast \longrightarrow \mathbb{Z}$} by 
$$f^\ast(a^\ast):=f(a) \text{ for each } a\in A.$$
Recall \cite{diestel} that $f:A \longrightarrow \mathbb{Z}$ is a coflow in $D$ if and only if $f^\ast$ is a flow in $D^\ast$.\\
We also transfer this relation to NL-flows and NL-coflows.

\begin{theorem}
Let $D=(V,A)$ be a plane digraph. A map $g: A \longrightarrow \mathbb{Z}$ is an NL-coflow in $D$ if and only if the map $g^\ast$ is an NL-flow in $D^\ast$.
\end{theorem}

\begin{proof}
Deleting an arc $a \in A$ corresponds to contracting $a^\ast \in A^\ast$. Deleting a set $S \subseteq A$ of arcs until $D-S$ is acyclic corresponds to contracting a set $S^\ast \subseteq A^\ast$ until all connected components of $D^\ast / S^\ast$ are strong.\\
Hence, a feedback arc set $S \subseteq A$ corresponds to a dijoin $S^\ast \subseteq A^\ast$. Thus a coflow $g$ in $D$ whose support contains a feedback arc set corresponds to a flow $f:=g^\ast$ in $D^\ast$ whose support contains a dijoin.
\end{proof}

Finally, we are able to state Neumann-Lara's conjecture as
\begin{conjecture}
Any loopless planar digraph admits an NL-$2$-coflow.
\end{conjecture}

\section{The NL-flow polynomial}\label{section NLpoly}

In contrast to the definition given in \cite{hochstreport} we will present a definition here, where the flow polynomial of the underlying graph is not involved anymore. Both definitions fulfill the same purpose, that is counting NL-$G$-flows.

We have already seen that a flow is an NL-flow if and only if its support is a dijoin of the digraph. This will be the basic observation throughout this section.

In order to develop a closed formula we use a kind of generalization of the well-known
inclusion-exclusion formula, the M\"obius inversion (see e.g.\ \cite{aigner}).

\begin{definition}
  Let $(P,\leq)$ be a finite poset, then the {\em M\"obius function}
  is defined as follows
$$ \mu: P \times P \rightarrow \mathbb{Z},\; \mu(x,y):= 
\begin{cases} 0 &\; ,\text{ if } x\nleq y \\
			  1 &\; ,\text{ if }  x=y\\
			  -\sum_{x \leq z < y} \mu(x,z) &\; ,\text{ otherwise }.
\end{cases}$$
\end{definition}

\begin{proposition}[see \cite{aigner}]\label{möbius inversion}
Let $(P,\leq)$ be a finite poset, $f,g : P \rightarrow \mathbb{K}$ functions and $\mu$ the M\"obius function. Then the following equivalence holds
$$ f(x)=\sum_{y \geq x}g(y) , \text{ for all } x \in P \Longleftrightarrow g(x)=\sum_{y \geq x}\mu(x,y)f(y), \text{ for all } x \in P.$$
\end{proposition}

With this so called \textit{M\"obius inversion from above} it will suffice to count all $G$-flows in some given subdigraphs. The next Lemma will be crucial not only for this purpose.

\begin{lemma}\label{crucial}
Let $G$ be an Abelian group, $M \in \lbrace 0, \pm 1 \rbrace^{m \times n}$ a totally unimodular matrix of full row rank and $b \in G^m$. Then the number of solutions of $Mx=b$ is $\vert G \vert^{(n-m)}$.
\end{lemma}

\begin{proof}
Choose a basis $B$ of $M$. Then
\begin{center}
\begin{tabular}{crcl}
                & $Mx$            &$=$  & $b$ \\
$\Leftrightarrow$ & $M^{-1}_{.B}Mx$ &$=$  & $M^{-1}_{.B}b$ \\
$\Leftrightarrow$ &$(I_m,\tilde{M}) \begin{pmatrix}
x_B \\ x_{\lbrace 1,...,n \rbrace \setminus B}
\end{pmatrix}$ &$=$ & $M^{-1}_{.B}b,$
\end{tabular}
\end{center}
where $\tilde{M}$ is a totally unimodular $(m \times (n-m))$-matrix. Thus, for every choice of values for the columns of $\tilde{M}$ we get exactly one solution of the equation.
\end{proof}

The basic observation that a flow is an NL-flow iff its support is a
dijoin encourages to consider the following poset $(\mathcal{C},\supseteq)$
$$\mathcal{C}:= \big \lbrace A \setminus  C \; \vert \; \exists \; C_1,...,C_r \text{ directed cuts }\text{, such that } C=\bigcup_{i=1}^r C_i \big \rbrace.$$

\begin{definition}\label{defflow}
Let $D=(V,A)$ be a digraph and $\mu$ the M\"obius function. Then the NL-flow polynomial of $D$ is defined as
$$\phi_{NL}^D(x):= \sum_{B \in \mathcal{C}}\mu(A,B)x^{\vert B \vert -rk(B)}.$$
\end{definition}

\begin{theorem}\label{NLflowTheo}
The number of NL-$G$-flows of a digraph $D$ depends only on the order $k$ of $G$ and is given by $\phi_{NL}^D(k)$.
\end{theorem}

\begin{proof}
Using Proposition \ref{möbius inversion} with 
$f_k,g_k: 2^A \rightarrow \Z$, such that $f_k(B)$ indicates all $G$-flows
and $g_k(B)$ all NL-$G$-flows in the subgraph of $D$ induced by $B \subseteq A$,
it suffices to show that
\begin{flalign}\label{möbiusvor}
f_k(Z)=\sum_{Y \subseteq Z}g_k(Y)
\end{flalign}
holds for all $Z \in \mathcal{C}$. Then we obtain
\begin{flalign*}
\phi_{NL}^D(k)=g_k(A)&=\sum_{B \subseteq A} \mu(A,B)f_k(B)\\
&=\sum_{B \in \mathcal{C}}\mu(A,B)k^{\vert B \vert - rk(B)},
\end{flalign*}
since the number of $G$-flows on $D[B]$ is given by $k^{\vert B \vert - rk(B)}$ due to Lemma \ref{crucial}.\\
Concerning (\ref{möbiusvor}) let $Z \in \mathcal{C}$ and $\varphi$ be a $G$-flow on $D[Z]$. With $d$ we denote the number of dicuts in $D[Z]$ and set
$$Y := Z \setminus \bigcup_{i=1}^d \left\lbrace C_i \; \vert \; C_i \text{ is a directed cut in }D[Z]\text{ and } \exists z \in Z \text{ s.t. } \varphi|_{C_i}(z)=0 \right\rbrace.$$
Then clearly $Y \in \mathcal{C}$ and $\varphi|_{Y}$ is an NL-$G$-flow on $D[Y]$. So, \mbox{$f_k(Z)\leq \sum_{Y \subseteq Z}g_k(Y)$}.\\
The other direction is obvious since every NL-$G$-flow $g$ on $D[Y]$ with $Y \subseteq Z$ can be extended to a $G$-flow $\tilde{g}$ on $D[Z]$, setting $\tilde{g}(a):= 0_G$ for all $a \in Z-Y$.
\end{proof}

Considering duality, our NL-flow polynomial becomes the NL-coflow polynomial of $D$ which equals the chromatic polynomial in \cite{HARUTYUNYAN}, \cite{altenbokum}, counting all acyclic colorings in $D$ devided by $x$.

\section{Regular Oriented Matroids and the Equivalence Theorem}\label{equivtheo}

The equivalence theorem for nowhere-zero flows has been generalized to regular oriented matroids by Crapo \cite{crapo_tutte_1969} and Arrowsmith and Jaeger \cite{arrowsmith}. Like them, we can generalize our results of the previous sections to oriented regular matroids, obtain an analogue Equivalence Theorem and present a polynomial counting all integer NL-$k$-flows which differs from the one in Definition \ref{defflow}. \\
One of our main tools will be the following variant of Farkas' Lemma (see 3.4.4 (4P) in \cite{BLSW}):
\begin{theorem}\label{farkas}
Let $\mathcal{O}$ denote an oriented matroid on a finite set $E$ given by its set of covectors and $\mathcal{O}^\ast$ its dual. Let $E=P \dot\cup N \dot\cup \ast \dot\cup O$ be a partition of $E$ and $i_0 \in P$. Either there exists $X \in \mathcal{O}^\ast$ such that $i_0 \in supp(X) \subseteq P \cup N \cup \ast$, supp$(X) \cap P \subseteq X^+$ and $supp(X) \cap N \subseteq X^-$ or there exists $Y \in \mathcal{O}$ such that $i_0 \in supp(Y)\subseteq P \cup N \cup O$, supp$(Y) \cap P \subseteq Y^+$ and supp$(Y)\cap N \subseteq Y^-$, but not both.
\end{theorem}

\begin{definition}
Let $\mathcal{O}$ denote the set of covectors of an oriented matroid on a finite set $E$. We say that $\mathcal{O}$ is totally cyclic, if the all $+$-vector is in $\mathcal{O}^\ast$, i.e. it is a vector. $S \subseteq E$ is a dijoin, if $(Y \in \mathcal{O}\setminus \lbrace 0 \rbrace$ and $Y \succeq 0)$ implies $supp(Y) \cap S \neq \emptyset$, i.e. $S$ meets every positive cocircuit.
\end{definition}

\begin{proposition}
$S \subseteq E$ is a dijoin if and only if $\mathcal{O}/S$ is totally cyclic.
\end{proposition}

\begin{proof}
Set $P=E \setminus S$, $\ast=S$ and $N=O=\emptyset$. Since $S$ is a dijoin, there is no non-zero vector $Y \in \mathcal{O}$ such that supp$(Y) \subseteq P$ and \mbox{supp$(Y) \cap P=Y^+$}. Thus, by Theorem \ref{farkas} for every $e \in P$ there exists $X_e \in \mathcal{O}^\ast$ such that \mbox{$e \in $ supp$(X_e)$} and supp$(X_e)\cap P \subseteq X_e^+$. The composition of these vectors $X_{e_1}\circ ... \circ X_{e_r} \in \mathcal{O}^\ast$ is the all $+$-vector in $\mathcal{O}^\ast \setminus S$, which proves that $\mathcal{O}/S$ is totally cyclic. The other implication is an immediate consequence from Theorem \ref{farkas}.
\end{proof}

Now we can define NL-flows in the setting of oriented regular matroids:

\begin{definition}
Let $M \in \lbrace 0,\pm 1 \rbrace^{m \times n}$ be a totally unimodular matrix and let $\mathcal{O}$ be the corresponding regular matroid. An NL-flow in $\mathcal{O}$ is a vector $x \in \mathcal{O}$ such that $\mathcal{O} /$supp$(x)$ is totally cyclic.\\
For $k \geq 2$ an NL-$k$-flow is an NL-flow in $\mathcal{O}$ with
$$ x \in \lbrace 0, \pm 1,...,\pm (k-1) \rbrace^n.$$
If $G$ is an Abelian group of order $k$, then an NL-$G$-flow is an NL-flow with
$$ x \in G^n.$$
\end{definition}

It is clear that with this definition we can define the NL-flow polynomial as before and immediately obtain the equivalence of the first two statements in Theorem \ref{equiv.theorem} by Theorem \ref{NLflowTheo}. The crucial Lemma \ref{crucial} for the proof of Theorem \ref{NLflowTheo} dealt with totally unimodular matrices anyway.\\
The only implication we are left to verify for an equivalence theorem for NL-flows in regular oriented matroids is (i) implies (iii), since (iii) implies (i) is trivial by taking the integer flow mod $k$. The following Lemma suffices for that purpose. It could be deduced from Proposition 5 in \cite{arrowsmith}. We give a short proof for completeness.

\begin{lemma}
Let $M \in \lbrace 0, \pm 1 \rbrace^{m \times n}$ be a totally unimodular matrix and let $x$ denote a $\mathbb{Z}_k$-flow in the corresponding regular matroid $\mathcal{O}$, i.e.\ $Mx \equiv 0 \mod k$. Then there exists a $k$-flow $y \in \lbrace 0, \pm 1,...,\pm(k-1)\rbrace^n$ in $\mathcal{O}$ such that $y \equiv x \mod k$.
\end{lemma}

\begin{proof}
Choose $y \in \lbrace 0, \pm 1,...,\pm(k-1)\rbrace^n$ satisfying $y \equiv x \mod k$ that minimizes $\| My \|_1=\sum_{i=1}^m \vert (My)_i \vert$. We claim that $y$ must be as required. Assume not and set $Y^+:= \lbrace i \; \vert \; y > 0\rbrace, Y^0:=\lbrace i \; \vert \; y =0 \rbrace$, $Y^-:=\lbrace i \; \vert \; y < 0 \rbrace$, \mbox{$P:=\lbrace i \; \vert \; (My)_i > 0 \rbrace$} and $N:=\lbrace i \; \vert \; (My)_i < 0 \rbrace$. By assumption $P \cup N \neq \emptyset$. Without loss of generality we only consider the case $P \neq \emptyset$, switching signs yields the other case. Hence let $i_0 \in P$. There cannot exist $u \in \R^m$ such that $u_P \geq 0, u_{i_0}>0, u_N \leq 0, u^\top M_{Y^+} \leq 0$ and $u^\top M_{Y^-} \geq 0$ for the first three inequalities imply $u^\top My > 0$ and the last two $u^\top M y \leq 0$. Hence by Theorem \ref{farkas}, applied to the pair of oriented matroids defined by the kernel and the row space of the totally unimodular matrix $\tilde{M}:=(M,I_P,I_N)$, there exists $(\tilde{y}^\top, z_P^\top, z_N^\top)^\top$ with $\tilde{y}_{Y^+} \leq 0$, $\tilde{y}_{Y^-} \geq 0$, $\tilde{y}_{Y^0}=0$, $z_P \geq 0$, $z_{i_0}>0$, $z_N \leq 0$ such that $M \tilde{y}=-I_Pz_P-I_Nz_N.$ Since $M$ is totally unimodular we may assume that $\tilde{y},z_P$ and $z_N$ have entries in $\lbrace 0,1,-1 \rbrace$ only. Thus, $y+k \tilde{y} \in \lbrace 0,\pm1,...,\pm(k-1)\rbrace^n$ and $y+k\tilde{y} \equiv x \mod k$. But since $My$ is divisible by $k$ we have
$$ \|M(y+k \tilde{y})\|_1 = \|My-kI_Pz_P-kI_Nz_N\|_1 < \|My\|_1,$$
contradicting the choice of $y$. 
\end{proof}

By the discussion preceding the last Lemma this implies the Equivalence Theorem for NL-flows:
\begin{theorem}
Let $\mathcal{O}$ be an oriented matroid given by a totally unimodular matrix $M$. Let $k \geq 2$ and $G$ be an Abelian group of order $k$. Then the following conditions are equivalent:
\begin{itemize}
\item[(i)] There exists an NL-$\mathbb{Z}_k$-flow in $\mathcal{O}$.
\item[(ii)] There exists an NL-$G$-flow in $\mathcal{O}$.
\item[(iii)] There exists an NL-$k$-flow in $\mathcal{O}$. \hfill $\square$
\end{itemize}
\end{theorem}

In \cite{arrowsmith} it is shown that the number of integer NZ-$k$-flows is also given by a polynomial in $k$ which differs from the flow polynomial defined by algebraic NZ-flows. In the following we will show that this theorem generalizes to NL-flows as well. The next Proposition provides the main tool concerning this intent.

\begin{proposition}[Ehrhart, see \cite{triangulations}]\label{ehrhart}
Given a convex polytope $P$, whose vertices belong to $\Z^d$, and for a positive integer $t$, let $\psi_P(t)=\# (tP \cap \Z^d)$ denote the number of integer points in the dilated polytope $tP=\lbrace tx \vert x \in P \rbrace$. Then $\psi_P(t)$ is a polynomial in $t$.
\end{proposition}

\begin{theorem}
Let $M \in \lbrace 0, \pm 1 \rbrace^{m \times n}$ be a totally unimodular matrix with full row rank, $\mathcal{O}$ the corresponding regular matroid on the finite ground set $E$ and let $k \geq 2$. Then the number of NL-$k$-flows in $\mathcal{O}$ is a polynomial in $k$.
\end{theorem}

\begin{proof}
We pursue the same strategy as in the proof of Theorem \ref{NLflowTheo} and set \mbox{$f_k,g_k:2^E \longrightarrow \mathbb{Z}$} such that $f_k(F)$ counts all $k$-flows in $\mathcal{O}(F)$ and $g_k(F)$ all NL-$k$-flows respectively. As proven in Theorem \ref{NLflowTheo} $f_k(Z)= \sum_{Y \subseteq Z}g_k(Y)$ holds for all $Z \in \mathcal{C}$. With Proposition \ref{möbius inversion} in mind it now suffices to prove that $f_k(F)$ is a polynomial in $k$ for all $F \subseteq E$.\\
For $p\leq m, q \leq n$ let $\mathcal{F} \in \lbrace 0, \pm 1 \rbrace^{p \times q}$ denote the totally unimodular submatrix of $M$ with full row rank corresponding to $\mathcal{O}(F)$. 
Due to (\ref{Mf=0}) we are looking for the number of solutions of $\mathcal{F}x=0$, where $x \in \mathbb{Z}^{q}$ and 
$$-(k-1) \leq x_i \leq k-1$$ 
holds for all $1\leq i \leq q$. Without loss of generality let the first $p$ columns of $\mathcal{F}$ build a basis of $\R^{p}$, denoted by $\mathcal{F}_B$ and denote the other $q-p$ columns with $\mathcal{F}_N$. Analoguesly denote the first $p$ entries of $x$ with $x_B$, the others with $x_N$. Then

\begin{center}
\begin{tabular}{cccccc}
                  & $\mathcal{F}x$   & $=$ & $\mathcal{F}_Bx_B+\mathcal{F}_Nx_N$         & $=$ &$0$ \\
$\Leftrightarrow$ &        &     & $x_B + \mathcal{F}_B^{-1}\mathcal{F}_Nx_N$  & $=$ &$0$ \\
$\Leftrightarrow$ & $x_B$  & $=$ & $-\mathcal{F}_B^{-1}\mathcal{F}_Nx_N$       & $=:$&$- \tilde{\mathcal{F}}x_N,$\\
\end{tabular}
\end{center}

where $\tilde{\mathcal{F}} \in \lbrace 0, \pm 1 \rbrace^{p \times q-p}$ is totally unimodular.\\
Thus we are looking for the number of solutions $x_N \in \mathbb{Z}^{q-p}$ with
\begin{center}
\begin{tabular}{cccccr}
$-(k-1)$ & $\leq$ & $x_i$ & $\leq$ & $k-1$ & $\forall i \in N,$ \\
$-(k-1)$ & $\leq$ & $(-\tilde{\mathcal{F}}x_N)_i$ & $\leq$ & $k-1$ & $\forall i \in B,$\\
\end{tabular}
\end{center}
i.e.\ the number of the vertices of the following convex polytope
\begin{flalign*} 
Q(F):= \left\lbrace  x_N \in \Z^{q-p} \; \vert \; \begin{pmatrix}
I_N \\ -I_N \\ \tilde{\mathcal{F}} \\ - \tilde{\mathcal{F}}
\end{pmatrix} x_N \leq \begin{pmatrix}
k-1 \\ \vdots \\ k-1
\end{pmatrix} \in \mathbb{Z}^{2q} \right\rbrace =(k-1) P(F), 
\end{flalign*}
where
$$P(F):= \left\lbrace  \frac{x_N}{k-1} \in \Z^{q-p} \; \vert \; \begin{pmatrix}
I_N \\ -I_N \\ \tilde{\mathcal{F}} \\ - \tilde{\mathcal{F}}
\end{pmatrix} \frac{x_N}{k-1} \leq \begin{pmatrix}
1 \\ \vdots \\ 1
\end{pmatrix} \in \mathbb{Z}^{2q} \right\rbrace.$$

Since $M$ is totally unimodular, $\mathcal{F}$ is totally unimodular, as well. This implies that all vertices of $P(F)$ are integer and Proposition \ref{ehrhart} yields that the number of integer points in $Q(F)$ is given by a polynomial in $k$, namely $\psi_{P(F)}(k-1)$. Thus, the number of NL-$k$-flows 
$$g_k(E)=\sum_{F \subseteq E} \mu(E,F)f_k(F)=\sum_{F \subseteq E} \mu(E,F)\psi_{P(F)}(k-1)$$ 
is a polynomial in $k$, too.
\end{proof}

\section{Applications on orientations of tournaments}

\subsection{Complete acyclic digraphs}
As an application we examine complete acyclic digraphs $D=(V,A)$. Recall that all acyclic digraphs with $n\geq 1$ vertices are isomorphic, thus the NL-flow polynomial does not depend on the orientation of the given digraph. \\
Moreover acyclic digraphs allow a topological ordering (see \cite{bangjensen}), which is an ordering of the vertices $v_1,...,v_n$ of $D$ such that for every arc $(v_i,v_j)\in A$ we have $i < j$.\\
In the complete case this ordering is even unique since complete acyclic digraphs contain a hamiltonian path:

\begin{lemma}[see e.g. \cite{bangjensen}]\label{top.ordering}
Every complete acyclic digraph allows a unique topological ordering.\\
\end{lemma}

\begin{minipage}{\linewidth}

\begin{center}
\begin{tabular}{ccc}
\begin{tikzpicture}
  [scale=.4,auto=left]
  
  \node (n1) at (6.5,11) [circle,fill=black] {};
  \fill [black] (n1) circle (0cm) node [right] {$\;\; a$};
  \node (n2) at (4,7) [circle,fill=black] {};
  \fill [black] (n2) circle (0cm) node [left] {$b \;\; $};
  \node (n3) at (9,7) [circle,fill=black] {};
  \fill [black] (n3) circle (0cm) node [right] {$\;\; c$};
  \node (n4) at (4,3) [circle,fill=black] {};
  \fill [black] (n4) circle (0cm) node [left] {$e \;\; $};
  \node (n5) at (9,3) [circle,fill=black] {}; 
  \fill [black] (n5) circle (0cm) node [right] {$\;\; d$};
  
  \foreach \from/\to in {n1/n2,n1/n3,n1/n4,n1/n5,n2/n4,n3/n4,n5/n4,n2/n3,n3/n5,n2/n5}
  \path [->] (\from) edge node[right] {} (\to); 
  
\end{tikzpicture}&

$\;\; \longrightarrow \;\;$&

\begin{tikzpicture}
  [scale=.4,auto=left]
  
  \node (n1) at (1,1) [circle,fill=black] {};
  \fill [black] (n1) circle (0cm) node [right] {$\;\; a$};
  \node (n2) at (1,3) [circle,fill=black] {};
  \fill [black] (n2) circle (0cm) node [right] {$\;\; b$};
  \node (n3) at (1,5) [circle,fill=black] {}; 
  \fill [black] (n3) circle (0cm) node [right] {$\;\; c$};
  \node (n4) at (1,7) [circle,fill=black] {};
  \fill [black] (n4) circle (0cm) node [right] {$\;\; d$};
  \node (n5) at (1,9) [circle,fill=black] {}; 
  \fill [black] (n5) circle (0cm) node [right] {$\;\; e$};
  
  \foreach \from/\to in {n1/n2,n2/n3,n3/n4,n4/n5}
  \path [->] (\from) edge node[right] {$\curlyvee$} (\to);

\end{tikzpicture}

\end{tabular}
\end{center}

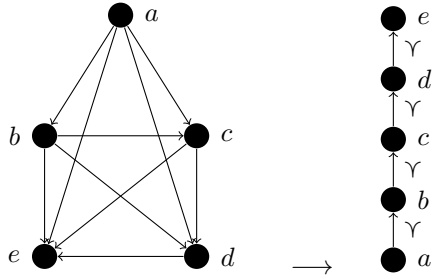
\captionof{figure}{topological ordering of $K_5$} \label{fig:M1}
\end{minipage}
\\
\\

Now, recall that a complete acyclic digraph with $n\geq 1$ vertices has exactly $n-1$ dicuts, in the following denoted by $C_1,...,C_{n-1}$. As a result the in Section \ref{section NLpoly} defined poset $(\mathcal{C},\supseteq)$ admits a simple structure.

\begin{proposition}
Let $D=(V,A)$ be a complete acyclic digraph with $\vert V \vert=n\geq 2$ and $(\mathcal{C},\supseteq)$ as above. Then $\mathcal{C}$ is isomorphic to $2^{[n-1]}$. 
\end{proposition}

\begin{proof}
Denote for some set $J$ of indices $C_J:= \cup_{j \in J} C_j$. Thus the elements of $\mathcal{C}$ are $A\setminus C_J$, for $J \subseteq [n-1]$ and the following map
\begin{flalign*}
\varphi : \mathcal{C} \rightarrow 2^{[n-1]} ,\; \varphi(A \setminus C_J):= J
\end{flalign*}
is well-defined since there are exactly $n-1$ dicuts. Moreover each set of indices $J \in 2^{[n-1]}$ induces exactly one element in $\mathcal{C}$, hence $\varphi$ is bijective. \\
Now, let $(A \setminus C_J )\supseteq (A \setminus C_I)$ for some $I,J \subseteq [n-1]$, thus $C_J \subseteq C_I$ and let $j \in J$. \\
Assume $j \notin I$. Then clearly $C_j \neq C_i$, for all $i \in I$, but $a \in C_I$ for each $a \in C_j$. Thus there are $i_1,...,i_s \in I$ all different from $j$ with $C_j \subseteq \bigcup_{k=1}^{s}C_{i_k}$. Hence 
\begin{flalign*} 
\# \text{comp}(D)+s &= \# \text{comp}(D \setminus (C_{i_1} \cup ... \cup C_{i_s})) \\
&= \# \text{comp}(D \setminus (C_{i_1} \cup ... \cup C_{i_s} \cup C_j))=\#\text{comp}(D)+s+1,
\end{flalign*}
where $\#$comp$(\cdot)$ counts the connected components. As a result of this contradiction we have $j \in I$ and $\varphi$ is an order isomorphism.
\end{proof}

As a result we get
\begin{flalign}\label{dicuts}
\phi_{NL}^D(k) &= \sum_{J \in \mathcal{C}}\mu (A,J)k^{\vert J \vert -rk(J)} \nonumber \\
&=\sum_{J \in 2^{[n-1]}} (-1)^{\vert J \vert} \; k^{\vert A \setminus \cup_{i \in J}C_i \vert - rk(A \setminus \cup_{i \in J}C_i)},
\end{flalign}
since $\mu(A,J)=(-1)^{\vert A \setminus J \vert}=(-1)^{\vert J \vert}$, for all $J \in 2^{[n-1]}$.
This immediately leads to the following theorem.

\begin{theorem}\label{theorem_acyclic}
Let $D=(V,A)$ be a complete acyclic digraph with $\vert V \vert =n$. For $1 \leq p \leq n$ denote by $(k_1,...,k_p)$ the composition of $n$ into $p$ parts, i.e. $\sum_{i=1}^p k_i=n$, with $k_i \geq 1$, $i=1,...,p$. Then the NL-flow polynomial is given by
$$ \phi_{NL}^D(x)= \sum_{p=1}^n (-1)^{p-1} \sum_{(k_1,...,k_p)} \prod_{i=1}^p x^{\binom{k_i -1}{2}}.$$
\end{theorem}

\begin{proof}
Let $n \geq 2$, otherwise we have $\phi^D_{NL}(x)=1$, the empty flow. For $J \in 2^{[n-1]}$ let $D[C_J]$ denote the subgraph of $D$ induced by $A\setminus \cup_{i \in J} C_i$ and $p=\vert J \vert +1$ the number of connected components in $D[C_J]$. We only have to count the number of arcs in $D[C_J]$, since the rank is given by $n-p$. \\
Deleting $\vert J \vert$ dicuts of the given complete digraph yields a subgraph with $p$ strongly connected components, each containing $k_i \geq 1$, $i=1,..,p$, vertices and thus $\binom{k_i}{2}$ arcs, satisfying $\sum_{i=1}^p k_i =n$.\\
Since the digraph is complete and acyclic, every combination is presumed, hence, with (\ref{dicuts}), the number of NL-$k$-flows is given by
$$ \sum_{p=1}^n (-1)^{p-1} \sum_{\substack{(k_1,...,k_p) \\ \sum_{i=1}^p k_i=n}} k^{\sum_{i=1}^p \binom{k_i}{2}-(n-p)}.$$
The claim follows, using $\binom{m}{2}-(m-1)=\binom{m-1}{2}$, for all $m \in \N$.
\end{proof} 

Now we can compute several NL-flow polynomials of complete acyclic digraphs with $n$ vertices in comparably short time:\\

\begin{tabular}{|C{1cm}|C{10.5cm}|}
\hline 
& \\ 
$n$ & $\phi_{NL}^D(x)$  \\ \hline
& \\ 
$1$ & $1$ \\ \hline
& \\ 
$2$ & $0$  \\ \hline
& \\ 
$3$ & $x-1$  \\ \hline
& \\ 
$4$ & $x^3-2x+1$  \\ \hline
& \\ 
$5$ & $x^6-2x^3+x$  \\ \hline
& \\ 
$6$ & $x^{10}-2x^6+x^3-x^2+2x-1$  \\ \hline
& \\ 
$7$ & $ {x}^{15}-2\,{x}^{10}+{x}^{6}-2\,{x}^{4}+2\,{x}^{3}+3\,{x}^{2}-4\,x+1$ \\ \hline
& \\
$8$ & ${x}^{21}-2\,{x}^{15}+{x}^{10}-2\,{x}^{7}+{x}^{6}+6\,{x}^{4}-4\,{x}^{3}
-3\,{x}^{2}+2\,x$ \\ \hline
\end{tabular}
\\
\\

Obviously there are a lot of regularities and we can explicitely give the exponent of the two leading terms and their coefficients. 

\begin{proposition}
Let $D=(V,A)$ be a complete acyclic digraph with $n\geq 1$ vertices.
  \begin{enumerate}
  \item[(i)] The leading term of $\phi_{NL}^D(x)$ equals $x^{\binom{n-1}{2}}$.
  \item[(ii)] Assume $n\geq 4$. Then the second term with highest exponent
    equals $-2x^{\binom{n-2}{2}}$.
  \end{enumerate}
\end{proposition}

\begin{proof}
We only need to consider the case where $p=1$, since the exponent of $\phi_{NL}^D(x)$ is maximum for $k_1=n$. The next lower exponent occurs when $p=2$, having $k_1=n-1$, $k_2=1$ and vice versa.   
\end{proof}

Let us now look at the constant term of the polynomial.

\begin{lemma}\label{recursion c(n)}
Let $D=(V,A)$ be a complete acyclic digraph with $n\geq 3$ vertices and $c(n)$ denote the constant term of $ \phi_{NL}^D(x)$. Then the following recursion holds
$$ c(n)=-(c(n-1)+c(n-2)).$$
\end{lemma}

\begin{proof}
Since we are interested in the constant term of $\phi_{NL}^D(x)$ we only need to consider the cases where $k_i \in \lbrace 1,2 \rbrace$ for all $1 \leq i \leq n$ and get the following distinction.
\begin{flalign*}
c(n)&=\sum_{\substack{k_2+...+k_p=n-1\\k_1=1\\k_i \in \lbrace 1,2 \rbrace}}(-1)^{p-1}+\sum_{\substack{k_2+...+k_p=n-2\\k_1=2\\k_i \in \lbrace 1,2 \rbrace}}(-1)^{p-1}\\
&\overset{r:=p-1}{=} -\sum_{\substack{k_1+...+k_r=n-1\\k_i \in \lbrace 1,2 \rbrace}}(-1)^{r-1} - \sum_{\substack{k_1+...+k_r=n-2\\k_i \in \lbrace 1,2 \rbrace}}(-1)^{r-1} \\
&= -( c(n-1)+c(n-2)).
\end{flalign*}
\end{proof}

This observation yields the following proposition.

\begin{proposition}\label{c(n)}
Let $D=(V,A)$ be a complete acyclic digraph with $n\geq 1$ vertices, then the constant term of $ \phi_{NL}^D(x)$ is given by
$$ c(n)=\phi_{NL}^D(0)=
\begin{cases}
-1 &\text{ , if n mod } 3 =0, \\
\;\;\;1  &\text{ , if n mod } 3 = 1, \\
\;\;\;0  &\text{ , if n mod } 3 =2.
\end{cases}$$
\end{proposition}

\begin{proof}
Lemma \ref{recursion c(n)} immediately yields 
$$ c(n+3)=-\big(c(n+2)+c(n+1)\big)=-\Big(-\big(c(n+1)+c(n)\big)+c(n+1)\Big)=c(n)$$
and the base cases from above prove the claim.
\end{proof} 

Observing the linear term we get:

\begin{proposition}
Let $D=(V,A)$ be a complete acyclic digraph with $n\geq 4$ vertices, then
the linear term of $ \phi_{NL}^D(x)$ is given by
$$l(n)=\frac{1}{3}
\begin{cases}
n &\text{ , if n mod } 3 =0, \\
-2(n-1)  &\text{ , if n mod } 3 = 1, \\
n-2  &\text{ , if n mod } 3 =2.
\end{cases}$$
\end{proposition}

\begin{proof}
In this case exactly one part of the composition, call it $k_j$, equals $3$, while the other parts have to be either $1$ or $2$. Let $c(n)$ be the constant term of $ \phi_{NL}^D(x)$, then we have
\begin{flalign*}
l(n)&= \sum_{\substack{k_1+...+k_{p-1}=n-1\\j\neq p\\k_i \in \lbrace 1,2 \rbrace, i\neq j}} (-1)^{p-1} 
+ \sum_{\substack{k_1+...+k_{p-1}=n-2\\j \neq p\\k_i \in \lbrace 1,2 \rbrace, i\neq j}} (-1)^{p-1} 
+ \sum_{\substack{k_1+...+k_{p-1}=n-3\\j=p\\k_i \in \lbrace 1,2 \rbrace}} (-1)^{p-1}\\
&=-l(n-1)-l(n-2)-c(n-3)
\end{flalign*}

Now we can proceed per induction, using Proposition \ref{c(n)}.
\begin{flalign*}
l(n+1)&=-l(n)-l(n-1)-c(n-2)\\
&\overset{IV}{=} -\frac{1}{3} \begin{cases}
n \\
-2(n-1)\\
n-2
\end{cases}
-\frac{1}{3}\begin{cases}
(n-1)-2\\
n-1\\
-2((n-1)-1)
\end{cases}
- \begin{cases}
1 &\text{ , if n mod } 3 = 0\\
0 &\text{ , if n mod } 3 = 1\\
-1 &\text{ , if n mod } 3 = 2
\end{cases}\\
&= \;\;\;\;\frac{1}{3} \begin{cases}
-2((n+1)-1) &\text{ , if n+1 mod } 3 = 1\\
(n+1)-1 &\text{ , if n+1 mod } 3 = 2\\
n+1 &\text{ , if n+1 mod } 3 = 0
\end{cases}
\end{flalign*}
\end{proof}

\subsection{Complete digraphs}

Considering an arbitrary complete digraph $D=(V,A)$ the NL-flow polynomial depends on its orientation. Let $d \in \mathbb{N}$ denote the number of maximal strongly connected components and denote their vertex sets with $S_1,...,S_d$. Since we cannot cut through cycles there are exactly $d-1$ dicuts and the poset $\mathcal{C}$ is isomorphic to $2^{[d-1]}$. Similarly as in (\ref{dicuts}) we conclude
\begin{flalign}\label{dicuts_complete}
\phi_{NL}^D(k) = \sum_{J \in 2^{[d-1]}} (-1)^{\vert J \vert} \; k^{\vert A \setminus \cup_{i \in J}C_i \vert - rk(A \setminus \cup_{i \in J}C_i)},
\end{flalign}
where $C_i, i=1,...,d-1$ denote the dicuts in $D$. \\
Recall that the maximal strongly connected components form a partition of the given digraph. Consequently we consider the following map, called condensation 
\begin{flalign*} 
\lambda : V &\rightarrow \lbrace 1,...,d \rbrace \\
v &\mapsto i, \text{ with }v \in S_i,
\end{flalign*}
which induces the complete acyclic digraph on $d$ vertices.\\
\\

\begin{minipage}{\linewidth}
\begin{tabular}{p{5cm}p{2cm}p{5cm}}
\begin{tikzpicture}
  [scale=.5,auto=left]
  
  \node (n1) at (1,5) [circle,fill=black!80] {};
  \node (n2) at (4,7) [circle,fill=black!40] {};
  \node (n3) at (8,7) [circle,fill=black!40] {};
  \node (n4) at (11,5)[circle,fill=black!80] {};
  \node (n5) at (8,3) [circle,fill=black!40] {};
  \node (n6) at (4,3) [circle,fill=black!40] {};
  
  \foreach \from/\to in {n1/n2,n1/n3,n1/n4,n1/n5,n1/n6,n2/n4,n3/n4,n5/n4,n6/n4,n2/n3,n3/n5,n5/n6,n6/n2,n2/n5,n6/n3}
  \path [->] (\from) edge [draw=black] (\to);
  
\end{tikzpicture}&
$$\overset{\lambda}{\longrightarrow}$$& 
\begin{tikzpicture}
  [scale=.5,auto=left]
  
  \node (n1) at (1,3) [circle,fill=black!80] {};
  \node (n2) at (4.5,7) [circle,fill=black!40] {};
  \node (n3) at (8,3) [circle,fill=black!80] {};  
  
  \foreach \from/\to in {n1/n2,n1/n3,n2/n3}
    \draw[->] (\from) -> (\to);   
\end{tikzpicture}
\end{tabular}

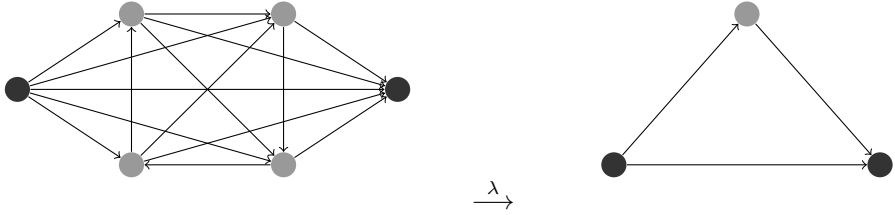
\captionof{figure}{condensation of a tournament}
\label{fig:M2}
\end{minipage}
\\
\\

As a result of Lemma \ref{top.ordering} the vertices of $D[\lambda(V)]$ can be ordered topologically, thus the strongly connected components of $D$ allow a similar ordering.

\begin{theorem}
Let $D=(V,A)$ be a complete digraph with $d\geq 1$ strongly connected components, each containing $k_1,...,k_d$ vertices, such that the subgraph of $D$ induced by $\lambda(V)$ is topologically ordered. For $1 \leq p \leq d$ consider the composition $(d_1,...,d_p)$ of $d$ into $p$ parts, i.e.\, $\sum_{i=1}^p d_i=d$, with $d_i \geq 1$, for all $1 \leq i\leq p$. Then the NL-flow polynomial is given by
\begin{flalign*}
\phi_{NL}^D(x)&= \sum_{p=1}^d (-1)^{p-1} \sum_{\substack{(d_1,...,d_p)}} \prod_{\substack{j=1}}^{p} x^{\binom{n_j-1}{2}},\text{ with} \\
n_j&:= \sum_{s=\delta(j-1)+1}^{\delta(j)}k_{s} \text{ and } \delta(j):=\sum_{r=1}^{j}d_r.
\end{flalign*}
\end{theorem}

\begin{proof}
Denote the strongly connected components of $D$ with $K_1,...,K_d$, such that the topologically ordering of $\lambda(V)$ is preserved. Analoguesly to the proof of Theorem \ref{theorem_acyclic} we only have to count the number of vertices in each partition of $D[\lambda(V)]$ induced by some composition $(d_1,...,d_p)$, where each vertex $1 \leq v \leq d$ corresponds to a strongly connected component $K_v$, each containing $k_v$ vertices in $D$. \\
So, let $(d_1,...,d_p)$ be an arbitrary composition of $d$ with $p$ parts, hence there are $d_j$, $1 \leq j \leq p$, vertices in each part of $D[\lambda(V)]$. Let $D_j$ denote the set of vertices in the corresponding strongly connected components in $D$. Then 
\begin{flalign*}
D_1= \bigcup_{i=1}^{d_1} K_i,\;
D_2= \bigcup_{i=d_{1}+1}^{d_1+d_2} K_i\;,..., \;
D_p=\bigcup_{i=d_1+...+d_{p-1}+1}^{d_1+...+d_{p}} K_i.
\end{flalign*}
Thus there are
$$\vert D_j \vert= \sum_{i=\sum_{r=1}^{j-1} d_{r}+1}^{\sum_{r=1}^{j}d_r} k_i$$ 
vertices in the $j$-th corresponding part of $D$.
\end{proof}

\section{Open Problems}

Whereas the computation of the flow polynomial of complete graphs seems to be quite hard in contrast to the computation of the chromatic polynomial of complete graphs, it turns out that we find the contrary in the directed case. We have no idea how the dichromatic polynomial of complete digraphs looks like while the computation of the NL-flow polynomial was not that challenging.\\

Considering the cographic oriented matroid our flow polynomial becomes the NL-coflow polynomial of $D$ which equals the chromatic polynomial \cite{altenbokum} for the dichromatic number divided by $x$. The natural question arises whether, as in the classical case, there exists a meaningful two variable polynomial combining both? Moreover, does such a polynomial or the two single variable polynomials have any meaning in the case of general oriented matroids?

\bibliography{NL-flow}{}
\bibliographystyle{siam}

\end{document}